\documentclass[11pt, oneside]{article}
\usepackage{geometry}                
\usepackage{graphicx}
\usepackage{amssymb, amsmath,amsthm}
\usepackage{epstopdf}
\usepackage{mathrsfs}
\usepackage{caption}
\usepackage{subfig}
  
\newtheorem{theorem}{Theorem}[section]
\newtheorem{corollary}{Corollary}[section]
\newtheorem{Lemma}{Lemma}[section]
\newtheorem{proposition}{Proposition}[section]

\theoremstyle{definition}

\newtheorem*{remark}{Remark}

\newcommand{\floor}[1]{\left\lfloor#1\right\rfloor}

\renewcommand{\d}{\, \text{d}}

\newcommand{\ep}{\varepsilon}

\title{The Finsler Metric Obtained as the $\Gamma$-limit of a Generalised Manhattan Metric}
\author{Hartmut Schwetlick\footnote{E-Mail: schwetlick@maths.bath.ac.uk}, Daniel C. Sutton\footnote{E-Mail: d@csutton.eu} and Johannes Zimmer\footnote{E-Mail: zimmer@maths.bath.ac.uk}\\Department of Mathematical Sciences\\
University of Bath\\
Bath BA2 7AY, U.K.}
\date{1st March 2013}                                           

\begin{document}

\maketitle

\begin{abstract}
The $\Gamma$-limit for a sequence of length functionals associated with a one parameter family of Riemannian manifolds is computed analytically. The Riemannian manifold is of `two-phase' type, that is, the metric coefficient takes values in $\{1,\beta\}$, with $\beta$ sufficiently large. The metric coefficient takes the value $\beta$ on squares, the size of which are controlled by a single parameter. We find a family of examples of limiting Finsler metrics that are piecewise affine with infinitely many lines of discontinuity. Such an example provides insight into how the limit metric behaves under variations of the underlying microscopic Riemannian geometry, with implications for attempts to compute such metrics numerically.
\end{abstract}

\section{Introduction}

We compute explicitly the $\Gamma(L^1)-$limit for the sequence of functionals
\begin{equation}\label{begin}
\int_0^1a_{\rho}\left(\frac{\gamma(\tau)}{\ep}\right)\|\gamma'(\tau)\| \d\tau, \;  \gamma \in W^{1,1}(0,1),
\end{equation}
where for $\rho \in (\tfrac{1}{2},1)$ the function $a_{\rho}$ is defined by
\begin{equation}\label{met123}
a_{\rho}(x,y) := 
\begin{cases}
\beta, & \text {if } (x,y) \in \tfrac{1}{2}(1 - \rho, 1 + \rho)^2\\
1, & \text{if } (x,y) \in [0,1]^2 \setminus \tfrac{1}{2}(1 - \rho,1 + \rho)^2,
\end{cases}
\end{equation}
extended periodically to $\mathbb R^2$. The value of $\beta$ is assumed to be fixed, the range of values will be determined later. The functional \eqref{begin} can be interpreted as the length functional for curves in a Riemannian manifold, for the metric co-efficient $a_{\rho}$. The $\Gamma(L^1)-$limit of such length functionals have been determined in the literature, see \cite{amar98a,buttazzo01a} for details. The main result is that the sequence $\Gamma(L^1)-$converges to a functional of the form
\begin{equation}\label{fins}
\int_0^1 \psi_{\rho}(\gamma'(\tau))\d \tau,
\end{equation}
where $\psi_{\rho}$ is convex and satisfies $\|\xi\| \leq \psi_{\rho}(\xi) \leq \beta\|\xi\|$, cf. \cite{amar98a}. In addition to this, $\psi_{\rho}$ is characterised by the \emph{asymptotic homogenisation formula},
\begin{equation}\label{asymp}
\psi_{\rho}(\xi) = \lim_{\ep \rightarrow 0} \inf \left\{\int_0^1a_{\rho}\left(\frac{\gamma(\tau)}{\ep}\right)\|\gamma'(\tau)\| \d\tau \colon \gamma \in W^{1,1}(0,1), \gamma(0) = 0, \gamma(1) = \xi \right\}.
\end{equation}
The focus of this study is to evaluate \eqref{asymp} for \eqref{begin}. The case $\rho=1$ has been previously calculated in \cite{c03,oberman09a}, and the limit $\psi_{1}(\xi)$ corresponds to the Manhattan norm. This is to be expected as on the microscopic scale one is confined to moving parallel to the $x$ or $y$ axis when the end points are in the region where $a_{\rho}(x,y) = 1$. It is in fact this property that ensures geodesics are easy to compute on the microscopic scale. Here we formulate a more general problem, that is, we allow our `streets' on the microscopic scale to have a non-trivial width, controlled by $\rho$, see figure \ref{a0}. An alternative interpretation of $\rho$ is that it controls the degree of obstruction imposed by the high cost regions. We can then study the impact of changing this microscopic information on the macroscopic description given by \eqref{asymp}.

The line of argument for evaluating \eqref{asymp} resembles \cite[Chapter 16]{braides98a}, where a checkerboard geometry is considered, with sufficiently high contrast to ensure that one set of squares may be entered, whereas the other set can not. The underlying microscopic features of the checkerboard metric make it easy to compute a geodesic by elementary geometric reasoning. In contrast, the problem considered in this paper has a geometry depending on a free parameter and a less restrictive underlying structure; it is thus unclear initially what a geodesic should be, we therefore need additional arguments to determine this. In particular, we reduce the infinite dimensional geodesic problem to a finite dimensional minimisation problem, based on several stages of geometric reasoning. We then solve the minimisation problem. It is note worthy to mention the work of \cite{amar09a}. Where also the checkerboard geometry is considered, but developed by an approach based on Snell's law, that has lower contrasts, where geodesics may begin to enter the higher contrast regions. This approach could be adapted to \eqref{met123} for $1 < \beta \leq 2$, given the additional considerations we make here, but does not help to evaluate \eqref{asymp} for $\beta > 2$. Examples of effective Hamiltonians for different metric geometries have been previously computed in \cite{acerbi84a,concordel97a,c03,braides98a}. The approach of these papers differs from the result here in the sense that unlike here, the metric coefficient is such that a geodesics can essentially be read off. 

To the best of our knowledge, no other example gives the homogenised limit as piecewise affine on infinitely many pieces, which may be an interesting unobserved phenomenon. Such an example may provide additional insight into the lower contrast checkerboard problem in \cite{amar09a}, where the authors experience difficulty in computing the full effective metric for $\beta$ close to one, but can compute the limit outside of the region where we find infinitely many likes of nondifferentiability accumulating.

This result seems to be the first to include a parameter that modifies the microscopic information, showing explicitly how this effects the macroscopic description given by \eqref{asymp}. The effect of varying $\rho$ can be seen in figure \ref{figur:2} in section \ref{sec3}. In particular we recover that the limit metric as $\rho$ tends to 1 produces the Manhattan metric. Additionally, the limit metric for $\rho \in (\tfrac{1}{2},1)$ produces infinitely many lines of discontinuity, therefore provides a significant challenge when trying to determine the limit metric numerically using methods as in \cite{gomes04a, oberman09a}.

Finally we mention two areas to which this example can be applied. The first is the minimisation of \eqref{begin} as the nonlinear Fermat's principle, where the values of $a_{\rho}$ define the refractive index of a optical material, as in \cite{amar09a}. For the model to hold, it is necessary to assume that the wave length of the light is much greater than the length scale $\ep$ and that we model only refractive light rays. The second application is connected to the propagation of a wave front though a heterogeneous media and the averaging of Hamiltonian dynamics as described in \cite{lions09a, c03, gomes01a, lions88a}. To see this connection first observe that by \cite{amar98a} it holds that \eqref{begin} $\Gamma(L^1)-$converges to \eqref{asymp} if and only if
\begin{equation}\label{esqua}
\int_0^1a_{\rho}\left(\frac{\gamma(\tau)}{\ep}\right)^2\|\gamma'(\tau)\|^2 \d\tau, \;  \gamma \in W^{1,2}(0,1)
\end{equation}
$\Gamma(L^2)-$converges to 
\begin{equation*}
\int_0^1 \psi_{\rho}(\gamma'(\tau))^2\d \tau.
\end{equation*}
The integrand of \eqref{esqua} may be interpreted as a Lagrangian, with corresponding Hamiltonian $H_{\ep}(p,x) = a_{\rho}\left(x/\ep\right)^2\|p\|^2$; a Hamiltonian related to the prorogation of wave fronts in heterogeneous media. The results of \cite{braides98a, evans92a,lions88a} state that solutions of the corresponding \emph{Hamilton-Jacobi} PDE
\begin{equation}\label{HJPDE}
\frac{\partial u}{\partial t} + H_{\ep}(\nabla_x u, x) = 0,
\end{equation}
subject to a suitable boundary condition on $u$, converge uniformly to
\begin{equation*}
\frac{\partial u}{\partial t} + \psi_{\rho}(\nabla_x u)^2 = 0
\end{equation*}
where we may think of $\psi_{\rho}^2$ as an effective Hamiltonian. Therefore our results provide insight into the effect that homogenisation has on the Hamilton-Jacobi PDE. In particular, our example has consequences for attempting to find the effective Hamiltonian by numerical methods as in \cite{gomes04a, oberman09a}. We also note that there is a connection between the regularity of the effective Hamiltonian and its corresponding solution as described in \cite{gomes01a}, the impact of this example on their work is left for future research. 

For notation, throughout we take $\mathbb N = \{1,2,3,...\}$, $|\cdot|$ the modulus function and $\| \cdot \|$ denotes the Euclidean norm on $\mathbb R^2$.
\subsection*{Acknowledgements} DCS was supported by an EPSRC Doctoral Training Account. The authors are grateful for funding from the network  ``Mathematical Challenges of Molecular Dynamics: A   Chemo-Mathematical Forum'' (EP/F03685X/1). 

\section{Characterisation of a class of geodesics for a single scale}\label{sec2}

\subsection{Reduction to shortest path problem on a finite discrete graph}\label{reduction}

In this section, we reduce the computation of a geodesic to that of a shortest path on a discrete graph. In this context a geodesic joining $(x_1,y_1)$ to $(x_2,y_2)$ is a curve $\gamma$, parameterised on $(0,1)$, minimising \eqref{begin} subject to $\gamma(0) = (x_1,y_1)$ and $\gamma(1) = (x_2,y_2)$. We compute a specific family of geodesics, for reasons outlined in section \ref{sec3}, using the length functional \eqref{begin}. In particular we determine a geodesic joining $ \left( \tfrac{1}{2}(1-\rho), -\tfrac{1}{2}(1-\rho) \right)$ to $ \left(M + \tfrac{1}{2}(1-\rho), N- \tfrac{1}{2}(1-\rho) \right)$ for $(M,N) \in \mathbb N^2$ with $M > N$. This is clearly equivalent to computing geodesics joining $(0,0)$ to $(M,N)$ in the shifted length functional
\begin{equation}\label{slen}
\int_0^1A_{\rho}(\gamma(\tau))\|\gamma'(\tau)\| \d\tau, \; \gamma \in W^{1,1}(0,1),
\end{equation}
where
\begin{equation*}
A_{\rho}(x,y) := a_{\rho}\left(x + \tfrac{1}{2}(1-\rho), y - \tfrac{1}{2}(1-\rho) \right).
\end{equation*}
For the remainder of this section we consider the latter minimisation problem, for some $M,N$ fixed, as the notation for this problem is less cumbersome. Let us define the sets $TL := (0,1) +  \mathbb Z^2$, $TR :=   (\rho, 1)+ \mathbb Z^2$, $BL :=   (0,1- \rho)+ \mathbb Z^2$, and $BR :=  (\rho, 1- \rho) + \mathbb Z^2$ corresponding to the top left/right and bottom left/right corners of the squares in $\Omega_{\text{g}}$ in the shifted metric, respectively. In addition, let $\Omega_{\text{g}}$ be the set of points $(x,y)$ where $A_{\rho}(x,y) = \beta$ and $\Omega_{\text{w}} := \mathbb R^2 \setminus \Omega_{\text{g}}$. See figure 1 for an illustration of the notation.

\begin{figure}[htbp]
\begin{center}
\includegraphics[scale = 1.15]{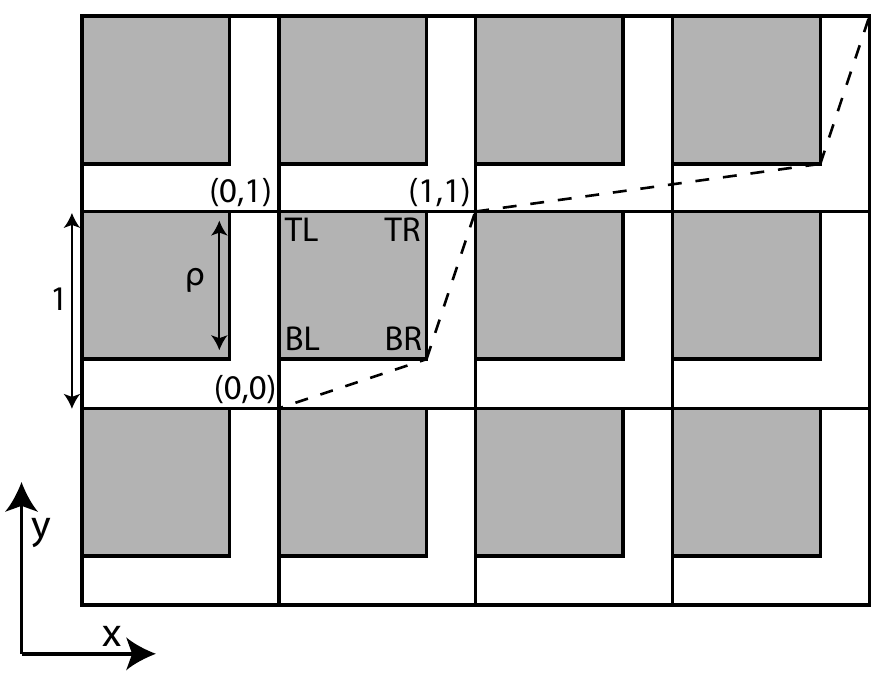}
\caption{Sketch of the shifted geodesic problem. Elements of the sets $TL$, $TR$, $BL$ and $BR$ are indicated. A geodesic for the shifted length functional joining $(0,0)$ to $(3,2)$ is shown. The shaded regions indicate $\Omega_{\text{g}}$.}
\label{a0}
\end{center}
\end{figure}

The length functional \eqref{begin} induces a metric on $\mathbb R^2$ by setting
\begin{multline}\label{dist}
d_{\ep}((x_1,y_1),(x_2,y_2)) =\\ \inf \left\{ \int_0^1a_{\rho}\left(\frac{\gamma(\tau)}{\ep}\right)\|\gamma'(\tau)\| \d\tau \colon \gamma \in W^{1,1}(0,1), \gamma(0)= (x_1,y_1), \gamma(1) = (x_2,y_2) \right\}.
\end{multline}
Recall that the integral in the definition \eqref{dist} may be reparameterised to another interval without changing the value of $d$. Furthermore, $d_{\ep}$ satisfies
\begin{equation}\label{euc}
 |(x_1,y_1)-(x_2,y_2)|\leq d_{\ep}((x_1,y_1),(x_2,y_2)) \leq \beta |(x_1,y_1)-(x_2,y_2)|.
\end{equation}

Since $d_{\ep}$ is uniformly equivalent to the Euclidean metric it follows that $(d_{\ep},\mathbb R^2)$ is complete, therefore by the Hopf-Rinow theorem \cite[Chapter 1]{jost05a} a geodesic exists for any given boundary conditions. The existence of geodesics for \eqref{slen} follows by identical considerations and we denote in this case, a geodesic joining $(0,0)$ to $(M,N)$ by $\gamma$. 

The following Lemma states that a geodesic joining $(0,0)$ to $(M,N)$ does not pass through $\Omega_{\text{g}}$, should the oscillation of $A_{\rho}$ be large enough. In addition, it restricts our attention to piecewise affine curves. 

\begin{Lemma}\label{nopass}
Any geodesic with endpoints in $\Omega_{\text{w}}$ does not pass through $\Omega_{\text{g}}$ for $\beta > 2$. Furthermore $\gamma$ is piecewise affine.
\end{Lemma}

\begin{proof}
Similar to \cite[Example 16.2]{braides98a}, or see \cite{suttontese} for a detailed proof for this particular case.
\end{proof}

For the remainder of this paper it is assumed that $\beta > 2$. The calculation when $1 < \beta < 2$ is more involved; an example of such a calculation for a checkerboard metric is the subject of \cite{amar09a}. We define $I := \{(x,y) \in \mathbb R^2 \colon \gamma(T) = (x,y), \; \lim_{\tau \rightarrow T^+} \gamma'(\tau) \neq \lim_{\tau \rightarrow T^-} \gamma'(\tau) \}$, that is, the points in $\mathbb R^2$ where a geodesic changes direction. The next Lemma shows that a geodesic only changes direction at the corners of $\Omega_{\text{g}}$.

\begin{Lemma}\label{corner}
It holds that
\begin{equation*}
\left(\mathbb R^2 \setminus \left( TL \cup TR \cup BL \cup BR \right)  \right) \cap I = \emptyset.
\end{equation*}
\end{Lemma}

\begin{proof}
Suppose the contrary. By Lemma \ref{nopass} any geodesic does not pass through $\Omega_{\text{g}}$, therefore given $x \in I$ it holds that $x \in \text{int}(\Omega_{\text{w}}) \cup \partial \Omega_{\text{w}}$. Suppose first that $x \in \text{int}(\Omega_{\text{w}})$, then there exists an open ball $\mathscr B_{r}(x) \subset \text{int}(\Omega_{\text{w}})$. Let $G$ be the connected component of $\text{Image}(\gamma) \cap \mathscr B_{r}(x)$ containing $x$ and let $T := \{ \tau \colon \gamma(\tau) \in G \}$. Set $s = \inf T$ and $t = \sup T$ and define
\begin{equation*}
v(\tau) := 
\begin{cases}
\displaystyle \frac{\gamma(t) - \gamma(s)}{t - s}(\tau - s) + \gamma(s) & \text{ if } \tau \in (s,t),\\
\gamma(\tau) & \text{ otherwise}.
\end{cases}
\end{equation*}

\begin{figure}[htbp]
\begin{center}
\includegraphics[scale = 0.9]{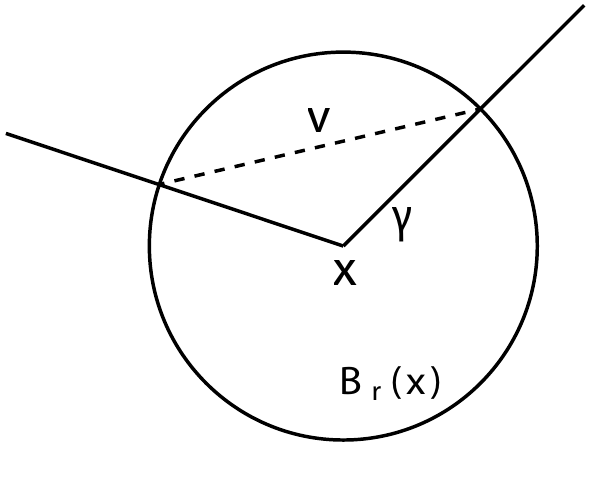}
\caption{Construction in Lemma \ref{corner}}
\label{a1}
\end{center}
\end{figure}

See figure \ref{a1} for an illustration of the construction. By construction $v \neq \gamma$ and
\begin{align*}\label{esta2}
\int_{s}^{t} a_{\rho}(v(\tau))\| v'(\tau) \| \d \tau <  \int_{s}^{t} a_{\rho}(\gamma(\tau))\| \gamma'(\tau) \| \d \tau,
\end{align*}
contradicting the minimality of $\gamma$. Now suppose that $x \in \partial \Omega_{\text{w}}$. Since $x$ by assumption is not at a corner of $\Omega_{\text{g}}$, there exists a half ball $\mathscr H_r(x)$ such that the flat edge is contained in $\partial \Omega_{\text{w}}$. Applying the previous argument to $\mathscr H_r(x)$ leads in a similar manner to the conclusion that $\gamma$ is not minimal.
\end{proof}

By Lemmas \ref{nopass} and \ref{corner} it follows that a geodesic consists of straight line segments joined at the corners of $\Omega_{\text{g}}$. The following Lemma reduces the number of potential geodesics to a finite set. 

\begin{Lemma}\label{containment}
The image of a geodesic joining $(0,0)$ to $(M,N)$ is contained in $[0,M] \times [0,N]$.
\end{Lemma}

\begin{proof}
Assume the contrary and suppose further that there exists a point $s \in (0,1)$ such that $\gamma_{1}(s) < 0$, the other cases are treated similarly. As $\gamma \in C^0(0,1)$ and since $\gamma(1) = (M,N)$, by the intermediate value theorem, there exists $t \in (s,1)$ such that $\gamma_{1}(t) = 0$, where $\ell_1$ denotes the first component of $\ell$. Define
\begin{equation*}
v(\tau) := 
\begin{cases}
\displaystyle \frac{\gamma(t)}{t}\tau & \text{ if } \tau \in (0,t),\\
\gamma(\tau) & \text{ otherwise}.
\end{cases}
\end{equation*}
As in Lemma \ref{corner} it follows that $v \neq \gamma$ and $\int_{0}^{t} a_{\rho}(v(\tau))\| v'(\tau) \| \d \tau  <  \int_{0}^{t} a_{\rho}(\gamma(\tau))\| \gamma'(\tau) \| \d \tau$,
contradicting the minimality of $\gamma$. 
\end{proof}

The next Lemma rules out some corners of $\Omega_{\text{g}}$ that a geodesic can pass through. More precisely Lemma \ref{TLBR} shows that a line segment starting at $TL$ must end in a set of $BR$ corners to the right and in the row above. 

\begin{Lemma}\label{TLBR}
Let $\ell \colon (s, t) \rightarrow \mathbb R^2$ be a maximal line segment of a geodesic such that $\ell(s) = (z_1,z_2)  \in TL$ where $z_1 \in \{1,...M-1\}$ and $z_2 \in \{1,...,N-1\}$. Then $ \ell(t) = (z_1 + Z - (1- \rho),z_2 + (1-\rho)) \in BR$ for $Z \in \{1,M-z_1\}$.
\end{Lemma}

\begin{proof}
The proof is split into three cases, depending on the angle at which the line segment leaves $TL$, denoted by $\theta \in [0,2\pi)$, where $\theta = 0$ is parallel to the $x$-axis.

\textit{Case 1: $\theta \in (\pi/2,2\pi)$}. It is clear that if $\theta \in (3\pi/2,2\pi)$ then the line segment would continue into $\Omega_{\text{g}}$, contradicting Lemma \ref{nopass}. It remains to rule out that $\theta \in (\pi/2,3\pi/2]$, which can be achieved using the same construction as in Lemma \ref{containment} to prove there exists a shorter curve.

\textit{Case 2: $\theta \in \{0, \pi/2\}$}. Suppose that $\theta = \pi/2$; the case $\theta = 0$ follows by a similar argument. In this case, $\gamma(s), \gamma(t) \in \{z_1\} \times [0,N]$. As $\gamma \in C^0(0,1)$ and since $\gamma(0) = (0,0)$, it follows that there exists $r \in (0,s)$ such that $\gamma_1(r) \in \{z_1-(1-\rho)\} \times [0,N]$. Therefore, applying the same reasoning as in Lemma \ref{corner}, we see that a geodesic must consist of straight line segments connecting $\gamma(r)$ to $\gamma(s)$ and $\gamma(s)$ to $\gamma(t)$. However, $\gamma(r), \gamma(s)$ and $\gamma(t)$ form a triangle in the set $[z_1-(1-\rho),z_1] \times [0,N]$. This contradicts the minimality of $\gamma$, see figure \ref{a4}.

\begin{figure}[htbp]
\begin{center}
\includegraphics[scale = 1.4]{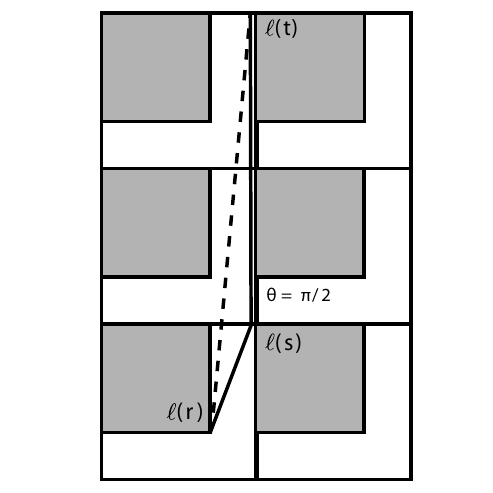}
\caption{Construction in Lemma \ref{TLBR} case 2. The vertical solid curve is the line segment $\ell$.}
\label{a4}
\end{center}
\end{figure}

\textit{Case 3: $\theta \in (0, \pi/2)$}.  Suppose first that the line segment connects $\ell(s)$ to any corner not stated in the Lemma, consequently $\ell_2(t) - \ell_2(s) \geq 1$, where $\ell_2$ is the second component of $\ell$. We prove, for $\rho \in (\tfrac{1}{2},1)$, should this line exist, then it necessarily crosses $\Omega_{\text{g}}$, contradicting Lemma \ref{nopass}. Consider the point $u \in (s,t)$ at which $\ell_2(s) +1= \ell_2(u)$, which exists by continuity. Then, either $\ell_1(u) \in (P, P+\rho)$ for a $P \in \{0,...,M-1\}$, in which case by continuity, $\ell(u-\delta) \in \Omega_{\text{g}}$ for $\delta$ sufficiently small. Alternatively, $\ell_1(u) \in [P+\rho,P+1]$ for a $P \in \{0,...M-1\}$. Parameterise $\ell$ over $(s,u)$ as a graph over the $x$-axis to obtain that $\ell_2(x) = x/\ell_1(u) + \ell_2(s)$ for $x \in (0,\ell_1(u))$. Evaluating $\ell_2$ at $x = P + \rho$ gives
\begin{equation*}
(1-\rho) + \ell_2(s) < \frac{P + \rho}{P+1} + \ell_2(s) \leq \frac{P + \rho}{\ell_1(u)} + \ell_2(s) \leq 1 + \ell_2(s),
\end{equation*}
if, and only if, $\rho \in (\tfrac{1}{2},1)$. Therefore by continuity, $\ell(u-\delta) \in \Omega_{\text{g}}$ for $\delta$ sufficiently small, a contradiction. It remains to rule out that the line segment ends at a BL corner in $W=[\ell_1(s),M]\times(\ell_2(s),\ell_2(s)+(1-\rho)]$. To rule out that the line segment ends in $BL$, repeat the reasoning of cases 1 and 2 for contradiction. Hence the line segment may only terminate at the $BR$ points of $W$ as stated in the theorem.
\end{proof}

Repeating the reasoning in Lemma \ref{TLBR} it is possible to show the analogous result for geodesics starting in $BR$. 

\begin{Lemma}\label{BRTL}
Let $\ell \colon (s, t) \rightarrow \mathbb R^2$ be a maximal line segment of a geodesic such that $\ell(s) =  (z_1 + \rho,z_2 +(1-\rho))   \in BR$ where $z_1 \in \{0,...M-1\}$ and $z_2 \in \{0,...,N-1\}$. Then $\ell(t) = (z_1+1,z_2+Z) \in TL$ for $Z \in \{1,N-z_2\}$.
\end{Lemma}

Lemmas \ref{TLBR} and \ref{BRTL} state should a geodesic lie in $(0,M) \times (0,N)$ then it necessarily joins points in $TL$ to $BR$ and then $BR$ to $TL$, in a specific way. We now show that we can extend this property further and rule out that a geodesic lies in $\partial \left( (0,M)\times(0,N) \right)$, except for the end points.

\begin{Lemma}\label{novert}
The image of a geodesic is contained in $(0,M)\times(0,N)$, except for the end points.
\end{Lemma}

\begin{proof}
Reasoning as in the proof of Lemma \ref{containment}, it is clear that should a geodesic have a line segment in $\partial \left( (0,M)\times(0,N) \right)$ then it must contain either $(0,0)$ or $(M,N)$, otherwise it is not minimal. Suppose that the line segment contains $(0,0)$, the other case is similar. Should the line segment end at $(0,N)$ then by Lemma \ref{containment} it must continue to join $(0,N)$ to $(M,N)$, giving a total length of $M+N$. However, choosing the curve joining $(0,0)$ to $(M-(1-\rho),\rho) \in BR$ and then onto $(M,N)$ is strictly shorter, therefore the longer curve is not a geodesic. Now suppose that the end of the line segment is $(0,Z) \in \{0\} \times \{1,...,N-1\}$ (otherwise by previous considerations, the curve is not a geodesic). Then by Lemma \ref{TLBR} a geodesic must extend as a line segment joining to a point of the form $(Y-(1-\rho),Z+(1-\rho)) \in BR \cap (0,M) \times (0,N)$ for $Y \in \{1,M\}$. Now consider the curve that first joins $(0,0)$ to $(Y-(1-\rho),1-\rho) \in BR \cap (0,M) \times (0,N)$, and then continues onto $(Y-(1-\rho),Z+(1-\rho))$, see figure \ref{a12}. 

\begin{figure}[htbp]
\begin{center}
\includegraphics[scale = 1.2]{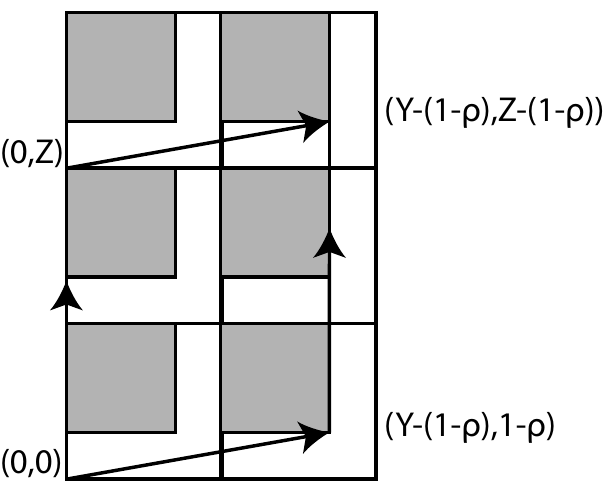}
\caption{Construction in Lemma \ref{novert}.}
\label{a12}
\end{center}
\end{figure}

Elementary geometric reasoning shows that the first two line segments of these curves share the same length, and that they both lie in $\Omega_{\text{w}}$. However, the latter curve contains a line segment parallel to the $y$-axis which is forbidden by Lemma \ref{BRTL} and therefore the curve cannot be minimal. 
\end{proof}

We can now identify potential geodesics by a pair of $k$-tuples. The length of each curve can then be described as a function of those $k$-tuples. One $k$-tuple records the distances $Z$ taken by applying Lemma \ref{TLBR}, the other $k$-tuple records the distances $Z$ from Lemma \ref{BRTL}. Since $(0,0) \in TL$ and $(M,N) \in TL$ and $TL$ connects to $BR$ which connects to $TL$ by Lemmas \ref{TLBR} and \ref{BRTL}, it suffices to record such $Z$ to describe the entire structure of the remaining curves.

\begin{Lemma}\label{len12}
The remaining candidate curves can be identified by $k$-tuples $(m_i)_{i=1}^{k}$, $(n_i)_{i=1}^k$ with  $\sum_{i=1}^k m_i = M$,$\sum_{i=1}^k n_i = N$. The length of a curve with such an identification is
\begin{equation}\label{lenfnl}
L\left[(m_i)_{i=1}^{k}, (n_i)_{i=1}^k  \right] = \sum_{i=1}^k \sqrt{(1-\rho)^2+(m_i - (1 - \rho))^2} + \sqrt{(1-\rho)^2+(n_i - (1 - \rho))^2}.
\end{equation}
Furthermore, $k \leq N$.
\end{Lemma}

\begin{proof}
Starting at $(0,0)$, by Lemma \ref{novert} and reasoning as in Lemma \ref{TLBR} the candidate geodesic must extend as a line segment joining a point of the form $  ( Z_1 - (1- \rho),1-\rho) \in BR \cap (0,M) \times (0,N)$ for some $Z_1 \in \{1,...,M\}$. This produces a length contribution of $\sqrt{(1-\rho)^2+(m_1 - (1 - \rho))^2}$, where $m_1 := Z_1$. Since $(M,N) \in TL$, the curve has not yet reached the end point. Therefore, applying Lemma \ref{BRTL}, the candidate geodesic continues as another line segment, connecting to $(m_1,Z_2) \in TL \cap (0,M] \times (0,N]$ for $Z_2 \in \{1,...,N\}$. The contribution to length is $\sqrt{(1-\rho)^2+(n_1 - (1 - \rho))^2}$, where $n_1: = Z_2$. Now, either $(m_1,n_1) = (M,N)$, in which case we terminate the procedure, or otherwise we may find $m_2 \in \{1,...,M-m_1\}$ and $n_2 \in \{1,...,N-n_1\}$, and so on until $\sum_{i=1}^k m_i = M$, $\sum_{i=1}^k n_i = N$. The procedure obviously ends after $k \leq N$ steps, otherwise we would contradict Lemma \ref{novert}.
\end{proof}

The results of this section have demonstrated that a geodesic is reduced to minimising \eqref{lenfnl} over $k$-tuples in
\begin{equation}\label{w3}
\left\{ (m_i)_{i=1}^k, (n_i)_{i=1}^k \in \mathbb N^k \colon k \leq N, \sum_{i=1}^k m_i = M,  \sum_{i=1}^k n_i = N \right\}. 
\end{equation}

Clearly this finite dimensional minimisation problem has a solution.

\subsection{Minimisation of the length functional}

This subsection is dedicated to the calculation of minima for \eqref{lenfnl} over $k$-tuples in \eqref{w3}. For notational convenience set
\begin{equation}\label{lengthhh}
\ell_{\rho}(x) := \sqrt{(1-\rho)^2+(x - (1 - \rho))^2}.
\end{equation}
To perform this minimisation, we first minimise \eqref{lenfnl} for fixed $k$ and then minimise over $k$. Lemmas \ref{mono2}, \ref{min1} and \ref{min2} are technical results to minimise \eqref{lenfnl} for fixed $k$. Denote by $\floor{\cdot}$ the floor function.

\begin{Lemma}\label{mono2}
For $x \in [1,\infty)$ and $\rho \in (\tfrac{1}{2},1)$, $\ell_{\rho}$ is strictly monotone increasing.
\end{Lemma}

\begin{proof}
A trivial calculus exercise.
\end{proof}

Lemmas \ref{min1} and \ref{min2} show that, for fixed $k$, \eqref{lenfnl} is minimised by distributing the values of the $k$-tuple equally. Note that the conditions of Lemma \ref{ineq5} ensure that $|z_1-z_2| \geq 2$.

\begin{Lemma}\label{ineq5}\label{min1}
For $z_1, z_2 \in \mathbb N$, with $2|(z_1+z_2)$, $z_1 \neq z_2$,
\begin{equation}\label{ineq1}
\ell_{\rho}(z_1) + \ell_{\rho}(z_2) > 2 \,\ell_{\rho}\left(\frac{z_1+z_2}{2}\right).
\end{equation}
\end{Lemma}

\begin{proof}
Suppose without loss of generality that $z_1 \geq (z_1+z_2)/2$ and $z_2 \leq (z_1+z_2)/2$. Observe that by the fundamental theorem of calculus \eqref{ineq1} holds if and only if 
\begin{equation}\label{ineq2}
\int_{({z_1+z_2})/{2}}^{z_1} \frac{d\ell_{\rho}}{dx}(x) \d x  - \int^{({z_1+z_2})/{2}}_{z_2} \frac{d\ell_{\rho}}{dx}(x) \d x > 0.
\end{equation}
An elementary calculation shows that
\begin{equation}\label{mono1}
\frac{d^2\ell_{\rho}}{dx^2}(x)= \frac{(1-\rho)^2}{\sqrt{(1-\rho)^2+(x - (1 - \rho))^2}} > 0,
\end{equation}
for $x \in [1,\infty)$. Thus, by strict monotonicity,
\begin{align*}
\int_{({z_1+z_2})/{2}}^{z_1}\frac{d\ell_{\rho}}{dx}(x) \d x & > \left( \frac{z_1-z_2}{2} \right) \frac{d\ell_{\rho}}{dx}\left(\frac{z_1+z_2}{2}\right), \\
\int^{({z_1+z_2})/{2}}_{z_2} \frac{d\ell_{\rho}}{dx}(x) \d x & < \left( \frac{z_1-z_2}{2} \right) \frac{d\ell_{\rho}}{dx}\left(\frac{z_1+z_2}{2}\right).
\end{align*}
Hence, \eqref{ineq2} and therefore \eqref{ineq1} holds.
\end{proof}

\begin{Lemma}\label{min2}
For $z_1,z_2 \in \mathbb N$, with $2 \nmid (z_1+z_2)$, $|z_1 - z_2| \geq 2$,
\begin{equation}\label{ineq4}
\ell_{\rho}(z_1) + \ell_{\rho}(z_2)  > \ell_{\rho}\left(\floor{\frac{z_1+z_2}{2}} \right)  + \ell_{\rho}\left(\floor{\frac{z_1+z_2}{2}} +1 \right).
\end{equation}
\end{Lemma}

\begin{proof}
Suppose without loss of generality that $z_1 > z_2$. First consider the case when 
\begin{align*}
C_1(z_1,z_2) &:= z_1 - \floor{\frac{z_1+z_2}{2}}+1 > 0,\\
C_2(z_1,z_2) &:= \floor{\frac{z_1+z_2}{2}} - z_2 > 0.
\end{align*}
Observe that \eqref{ineq4} holds if, and only if, 
\begin{equation*}
\int_{\floor{({z_1+z_2})/{2}}+1}^{z_1}\frac{d\ell_{\rho}}{dx}(x) \d x  - \int^{\floor{({z_1+z_2})/{2}}}_{z_2} \frac{d\ell_{\rho}}{dx}(x) \d x > 0.
\end{equation*}
Then, by strict monotonicity, using \eqref{mono1},
\begin{align*}
\int_{\floor{({z_1+z_2})/{2}}+1}^{z_1} \frac{d\ell_{\rho}}{dx}(x) \d x & > C_1(z_1,z_2)\frac{d\ell_{\rho}}{dx}\left(\floor{\frac{z_1+z_2}{2}}+1 \right), \\
\int^{\floor{({z_1+z_2})/{2}}}_{z_2} \frac{d\ell_{\rho}}{dx}(x) \d x & < C_2(z_1,z_2)\frac{d\ell_{\rho}}{dx}\left(\floor{\frac{z_1+z_2}{2}} \right).
\end{align*}
The claim follows once we have shown that
\begin{equation}\label{ineq3}
C_1(z_1,z_2)\frac{d\ell_{\rho}}{dx}\left(\floor{\frac{z_1+z_2}{2}}+1 \right) - C_2(z_1,z_2)\frac{d\ell_{\rho}}{dx}\left(\floor{\frac{z_1+z_2}{2}} \right)> 0.
\end{equation}
By monotonicity, from \eqref{mono1}, the left hand side of \eqref{ineq3} is strictly greater than
\begin{equation*}
(C_1(z_1,z_2)-C_2(z_1,z_2))\frac{d\ell_{\rho}}{dx}\left(\floor{\frac{z_1+z_2}{2}} \right).
\end{equation*}
Since $C_1(z_1,z_2) - C_2(z_1,z_2) = z_1+z_2 - 2\floor{(z_1+z_2)/2} + 1 > 1$ and $\floor{(z_1+z_2)/2} \geq 1$, it follows that \eqref{ineq4} holds. 

The case $C_1(z_1,z_2) = C_2(z_1,z_2) = 0$ is impossible by our assumption that $|z_1-z_2| \geq 2$. Since $2 \nmid (z_1+z_2)$ the cases $C_1(z_1,z_2) = 0, C_2(z_1,z_2) \neq 0$ and $C_2(z_1,z_2) = 0, C_1(z_1,z_2) \neq 0$ also do not arise.
\end{proof}

We now minimise \eqref{lenfnl} over \eqref{w3} subject to $k \leq N$ fixed. 

\begin{Lemma}\label{soln1}
Let $1 \leq k \leq N$, then we can write $M = \ell_1k+s_1$, $N=\ell_2k+s_2$ for $\ell_i, s_i \in \mathbb N$. Then a pair of $k$-tuples $(m_i)_{i=1}^k, (n_i)_{i=1}^k \in \mathbb N^k$ that minimises the functional
\begin{equation*}
\sum_{i=1}^k \ell_{\rho}(m_i) +\ell_{\rho}(n_i)
\end{equation*}
subject to
\begin{equation}\label{cons1}
\sum_{i=1}^k m_i = M, \qquad \sum_{i=1}^k n_i = N
\end{equation}
takes the form $m_i = \ell_1$ for $k-s_1$ terms, $m_i = \ell_1+1$ for $s_1$ terms, $n_i = \ell_2$ for $k-s_2$ terms and $n_i = \ell_2+1$ for $s_2$ terms.
Furthermore, this solution is unique, up to permutations.
\end{Lemma}

\begin{proof}
Suppose, without loss of generality, that the $k$-tuple $(m_i)_{i=1}^k$ is not of the form $m_i = \ell_1$ for $k-s_1$ terms and $m_i = \ell_1+1$ for $s_1$ terms. Then by constraint \eqref{cons1}, there exists at least two terms of the $k$-tuple $m_1, m_2$ such that $|m_1 - m_2| \geq 2$. 

If $2 | (m_1+m_2)$, then by Lemma \ref{min1} it holds that
\begin{equation*}
\ell_{\rho}(m_1) + \ell_{\rho}(m_2) > 2 \, \ell_{\rho}\left(\frac{m_1+m_2}{2}\right),
\end{equation*}
contradicting the minimality of the proposed solution. Otherwise $2 \nmid (m_1+m_2)$, so that by Lemma \ref{min2} 
\begin{equation*}
\ell_{\rho}(m_1) + \ell_{\rho}(m_2)  > \ell_{\rho}\left(\floor{\frac{m_1+m_2}{2}}\right) + \ell_{\rho}\left(\floor{\frac{m_1+m_2}{2}} + 1\right),
\end{equation*}
again contradicting the minimality of the proposed solution. The uniqueness up to rearrangement of indices follows from the uniqueness of the representations $M = \ell_1k+s_1$, $N=\ell_2k+s_2$. Hence the result holds.
\end{proof}

With a minimiser for each $k$ found, it remains to minimise over $k$. To achieve this, it suffices to show that increasing $k$ strictly reduces length. Lemmas \ref{sing1} and \ref{sing2} show replacing the $k$-tuple with a $k+1$-tuple leads to a strict reduction in length.

\begin{Lemma}\label{sing1}
Let $z_1 \in \mathbb N$, suppose $2|z_1$ and $z_1 \geq 2$, then
\begin{equation}\label{sinineq1}
\ell_{\rho}(z_1) > 2 \, \ell_{\rho} \left(\frac{z_1}{2}\right).
\end{equation}
\end{Lemma}

\begin{proof}
Since $2|z_1$, write $z_1 = 2k$ for some $k \in \mathbb N$. Then, \eqref{sinineq1} is equivalent to showing that
\begin{equation}\label{cond1}
\int_{k}^{2k} \frac{d\ell_{\rho}}{dx}(x) \d x - \ell_{\rho}(k) > 0.
\end{equation}
By monotonicity, from \eqref{mono1}, we have that 
\begin{equation*}
\int_{k}^{2k} \frac{d\ell_{\rho}}{dx}(x) \d x - \ell_{\rho}(k) > k\frac{d\ell_{\rho}}{dx}(k) - \ell_{\rho}(k).
\end{equation*}
It is easy to verify that
\begin{equation}\label{linineq}
k\frac{d\ell_{\rho}}{dx}(k) - \ell_{\rho}(k) = \frac{(1-\rho)(k - 2(1-\rho))}{\sqrt{(1-\rho)^2+(k - (1 - \rho))^2}} = \frac{(1-\rho)(k - 2(1-\rho))}{\ell_{\rho}(k)}. 
\end{equation}
Furthermore, since $\ell_{\rho} > 0$, it holds that the right hand side of \eqref{linineq} is positive for $k \in \mathbb N$. Hence \eqref{cond1} holds.
\end{proof}

\begin{Lemma}\label{sing2}
Let $z_1 \in \mathbb N$, suppose $2 \nmid z_1$ and $z_1 \geq 2$, then
\begin{equation}\label{sinineq2}
\ell_{\rho}(z_1) > \\ \ell_{\rho}\left(\floor{\frac{z_1}{2}} \right) + \ell_{\rho}\left(\floor{\frac{z_1}{2}}+1 \right).
\end{equation}
\end{Lemma}

\begin{proof}
Since $2 \nmid z_1$, write $z_1 = 2k + 1$ for some $k \in \mathbb N$. Then, \eqref{sinineq2} is equivalent to showing that,
\begin{equation*}
\int_{k+1}^{2k+1} \frac{d\ell_{\rho}}{dx}(x) \d x - \ell_{\rho}(k) > 0.
\end{equation*}
By monotonicity, from \eqref{mono1}, we have that
\begin{equation*}
\int_{k+1}^{2k+1} \frac{d\ell_{\rho}}{dx}(x) \d x - \ell_{\rho}(k) > k\frac{d\ell_{\rho}}{dx}(k+1) - \ell_{\rho}(k) > k\frac{d\ell_{\rho}}{dx}(k) - \ell_{\rho}(k).
\end{equation*}
Hence continuing from \eqref{linineq} in Lemma \ref{sing1} completes the proof.
\end{proof}

The following Lemma combines Lemmas \ref{sing1} and \ref{sing2} to show that the minimal $k+1$-tuples have total length strictly shorter than the minimal $k$-tuples.

\begin{Lemma}\label{split}
Let $(z_i)_{i=1}^k$ and $(\tilde z_i)_{i=1}^{k+1}$ be a $k$-tuple and $k+1$-tuple with $z_i$ being a placeholder for either $m_i$ or $n_i$ as in Lemma \ref{soln1}. Then
\begin{equation}
\sum_{i=1}^k \ell_{\rho}(z_i) > \sum_{i=1}^{k+1} \ell_{\rho}(\tilde z_i).
\end{equation}
\end{Lemma}

\begin{proof}
Suppose that there exists $j \in \{1, ... , k\}$ such that $z_j \geq 2$; without loss of generality assume $j = k$. Define a new $k+1$-tuple by $\hat z_i = z_i$ if $i \in \{1,...,k-1\}$. If $2 | z_j$ then set $\hat z_k = \hat z_{k+1} = z_j/2$, otherwise set $\hat z_k =  \floor{z_j}/2$ and  $\hat z_{k+1} =  \floor{z_j}/2+1$. Using Lemmas \ref{sing1} or \ref{sing2}, it holds that
\begin{equation*}
\sum_{i=1}^k \ell_{\rho}(z_i) > \sum_{i=1}^{k+1} \ell_{\rho}(\hat z_i).
\end{equation*}
Furthermore, since $\sum_{i=1}^{k+1} \hat z_i =  \sum_{i=1}^{k+1} z_i $, by the minimality of $(\tilde z_i)_{i=1}^{k+1}$ we have that
\begin{equation*}
\sum_{i=1}^{k+1} \ell_{\rho}(\hat z_i) \geq \sum_{i=1}^{k+1} \ell_{\rho}(\tilde z_i).
\end{equation*}
Now consider the case when $z_i \equiv 1$ for all $i$. This implies that $k = N$, by Lemma \ref{soln1}, and hence there is no such $k+1$-tuple.
\end{proof}

From Lemma \ref{split}, it is possible to compute $\min L$ explicitly, and the corresponding geodesic curves.

\begin{proposition}\label{len1}
The length of a geodesic joining $(0,0)$ to $(M,N)$ is
\begin{multline}\label{len_ncase}
\mathscr L_{\rho}(M,N) := N\ell_{\rho}(1) + \left(M - \floor{{M}/{N}}N \right)\ell_{\rho}\left(\floor{{M}/{N}} + 1 \right) \\+ \left(N - M + \floor{{M}/{N}}N \right)\ell_{\rho}(\floor{{M}/{N}}).
\end{multline}
\end{proposition}

\begin{proof}
By Lemma \ref{split}, it is clear that taking $k = N$, with the corresponding $N$-tuple $(n_i)_{i=1}^{N}$ where $n_i = 1$ for all $i$ produces curves of minimal length. It follows that the corresponding $N$-tuple $(m_i)_{i=1}^{N}$ is also optimal. Writing $M = RN + S$, it holds that $m_i = R$ for $N-S$ terms and $m_i = R+1$ for $S$ terms. Hence, the minimal length is 
\begin{align*}
\mathscr L_{\rho}(M,N) &= N\ell_{\rho}(1) + S\ell_{\rho}(R+1) +  \left(N - S\right)\ell_{\rho}(R).
\end{align*}
Note that $S = M - \floor{{M}/{N}}N \text{ and } R =  \floor{{M}/{N}}$, which completes the proof.
\end{proof}

The curve of length \eqref{len_ncase} is not necessarily unique, as the following corollary shows.

\begin{corollary}
There are precisely $\binom{N}{M - \floor{M/N}N}$ geodesics joining $(0,0)$ to $(M,N)$.
\end{corollary}

\begin{proof}
The potential source of non-uniqueness stems from the fact that in Proposition \ref{len1}, the $N$-tuple $(m_i)_{i=1}^N$ is only unique up to a permutation. Hence the result follows. 
\end{proof}

The intuition behind this can be seen in figure \ref{a0}. It does not matter whether a geodesic first joins $TL$ to $BR$ over two squares and then the next connection $TL$ to $BR$ is one square, as can be seen in the figure. This non-uniqueness is reflected in the various permutations of $(m_i)_{i=1}^N$ that we can take.

The next subsection focuses on constructing a sequence of geodesics to compute the limit length.

\subsection{The $\ep$-scaled Riemannian length functional}\label{ss23}

The aim of this subsection is to compute a sequence of geodesics, denoted $\gamma_{\ep}$, for the scaled length functional \eqref{begin}. For the $\ep$-dependent problem we choose to compute geodesics joining
\begin{equation}\label{endpts}
\left( \ep \tfrac{1}{2}(1-\rho), -\ep \tfrac{1}{2}(1-\rho) \right) \text{ to } \left(M + \ep \tfrac{1}{2}(1-\rho), N- \ep \tfrac{1}{2}(1-\rho) \right)
\end{equation} for $(M,N) \in \mathbb N^2$ with $M > N$. 
As before, this is equivalent to computing geodesics joining $(0,0)$ to $(M,N)$ in the shifted length functional
\begin{equation}\label{aeshi}
\int_0^1A_{\rho}\left(\frac{\gamma(\tau)}{\ep}\right)\|\gamma'(\tau)\| \d\tau, \; \gamma \in W^{1,1}(0,1),
\end{equation}
where
\begin{equation*}
A_{\rho}(x,y) := a_{\rho}\left(x - \ep \tfrac{1}{2}(1-\rho), y - \ep \tfrac{1}{2}(1-\rho) \right).
\end{equation*}
For each $\ep > 0$, determining the minimal length of \eqref{aeshi} is an identical argument to when $\ep = 1$, except that all line segments are scaled by a factor $\ep$. Thus for a fixed $\ep$ that the length of a geodesic joining $(0,0)$ to $(\ep M, \ep N)$ in \eqref{aeshi} is $\ep L(M,N)$. Define $L_{\rho}^{\ep}(x,y)$ to be the length of a geodesic joining $(0,0)$ to $(x,y)$ in \eqref{aeshi}.

\begin{Lemma}\label{epk}
Let $(x,y) \in \mathbb Q^2$, $x > y > 0$, and suppose $x = p/q$, $y = r/s$. Then there exists a sequence $(\ep_k)_{k=1}^{\infty}$ with $\ep_k \rightarrow 0$ as $k \rightarrow \infty$ such that 
\begin{equation}
L_{\rho}^{\ep_k}(x,y) = L_{\rho}(x,y),
\end{equation}
where $L_{\rho}(x,y)$ is the extension of \eqref{len_ncase} to $\mathbb Q^2$.
\end{Lemma}

\begin{proof}
Take $\ep_k = 1/kqs$, $M = kps$ and $N = kqr$. Then by elementary geometric reasoning
\begin{equation}
L_{\rho}^{\ep_k}(x,y) = \frac{1}{kqs} L_{\rho}(kps,krq).
\end{equation}
It also holds that $\frac{1}{kqs} L_{\rho}(kps,krq)  = L_{\rho}(x,y)$ (to show this is a trivial calculation) therefore the result holds.
\end{proof}

\section{The limit metric}\label{sec3}

In this section we compute the limit metric corresponding to the $\Gamma$-limit of the sequence of functionals \eqref{begin}. 
\begin{Lemma}
Let $(x,y) \in \mathbb Q^2$, $x > y > 0$, and suppose $x = p/q$, $y = r/s$. Then the limit metric takes the value
\begin{equation}
\psi_{\rho}(x,y) = L_{\rho}(x,y).
\end{equation}
\end{Lemma}

\begin{proof}
By \eqref{asymp}
\begin{equation}
\psi_{\rho}(x,y) = \lim_{i \rightarrow \infty} L_{\rho}^{a,\ep_i}(x,y),
\end{equation}
where
\begin{equation}
L_{\rho}^{a,\ep}(x,y) = \inf \left\{ \int_0^1a_{\rho}\left(\frac{\gamma(\tau)}{\ep}\right)\|\gamma'(\tau)\| \d\tau \colon \gamma \in W^{1,1}(0,1), \gamma(0)= (0,0), \gamma(1) = (x,y) \right\}.
\end{equation}
Furthermore, the limit is independent of the choice of $(\ep_i)_{i=1}^{\infty}$ where $\ep_i \rightarrow 0$ as $i \rightarrow \infty$ by \cite[Proposition 15.5]{braides98a}. Using the triangle inequality for \eqref{dist} and \eqref{euc} we find
\begin{multline*}
\left| d_{\ep}((0,0),(x,y)) - d_{\ep}\left((- \ep \tfrac{1}{2}(1-\rho),- \ep \tfrac{1}{2}(1-\rho)),(x- \ep \tfrac{1}{2}(1-\rho),y- \ep \tfrac{1}{2}(1-\rho)) \right) \right| \\ \leq d_{\ep}((0,0),(- \ep \tfrac{1}{2}(1-\rho),- \ep \tfrac{1}{2}(1-\rho))) + d_{\ep}((x,y),(x- \ep \tfrac{1}{2}(1-\rho),y- \ep \tfrac{1}{2}(1-\rho)))\\
\leq C\ep
\end{multline*}
By definition \begin{align*}
d_{\ep}((0,0),(x,y)) & = L_{\rho}^{a,\ep}(x,y),\\
d_{\ep}\left((- \ep \tfrac{1}{2}(1-\rho),- \ep \tfrac{1}{2}(1-\rho)),(x- \ep \tfrac{1}{2}(1-\rho),y- \ep \tfrac{1}{2}(1-\rho)\right) & = L_{\rho}^{\ep}(x,y),
\end{align*}
So the last estimate implies $\lim_{\ep \rightarrow 0} L_{\rho}^{a,\ep}(x,y) = \lim_{\ep \rightarrow 0} L_{\rho}^{\ep}(x,y)$. Therefore, by taking $(\ep_k)_{k=1}^{\infty}$ as in Lemma \ref{epk} it holds that
\begin{equation*}
\psi_{\rho}(x,y) = \lim_{k \rightarrow \infty} L_{\rho}^{a,\ep_k}(x,y) =  \lim_{k \rightarrow \infty} L_{\rho}^{\ep_k}(x,y) = L_{\rho}(x,y),
\end{equation*}
by Lemma \ref{epk}.
\end{proof}

It is now possible to construct the limit metric $\psi_{\rho}$ on $\mathbb R^2$.

\begin{theorem}\label{main2}
The limit metric is given by
\begin{equation}
\psi_{\rho}(x,y) = L_{\rho}(\max \{|x|,|y|\},\min\{|x|,|y|\}).
\end{equation}
\end{theorem}

\begin{proof}
Use the fact that $\psi_{\rho}$ is continuous to extend to $(x,y) \in \mathbb R^2$, $x \geq y \geq 0$. To extend to other regions of $\mathbb R^2$, follow an identical procedure as before, applying rotations and reflections as necessary.
\end{proof}

Diagrams of the limit metric for different values of $\rho$ are given in figure \ref{figur:1}. The properties of $\psi_{\rho}$ are discussed in the next section.

\begin{figure}[htp]
  \centering
  \label{f1}
  \caption{Diagram of the the structure of the set $\{ x \in \mathbb R^2 : \psi_{\rho}(x) = 1\}$.  The dashed lines are lines of the form $y = \pm x/k$ for $k \in \mathbb N$. The lines of discontinuity accumulate at the $x$ and $y$ axis. The structure of $\psi_{\rho}$ on other quadrants is obtained by symmetry.}
  \subfloat[$\rho = 1$]{\label{figur:1}\includegraphics[width=70mm]{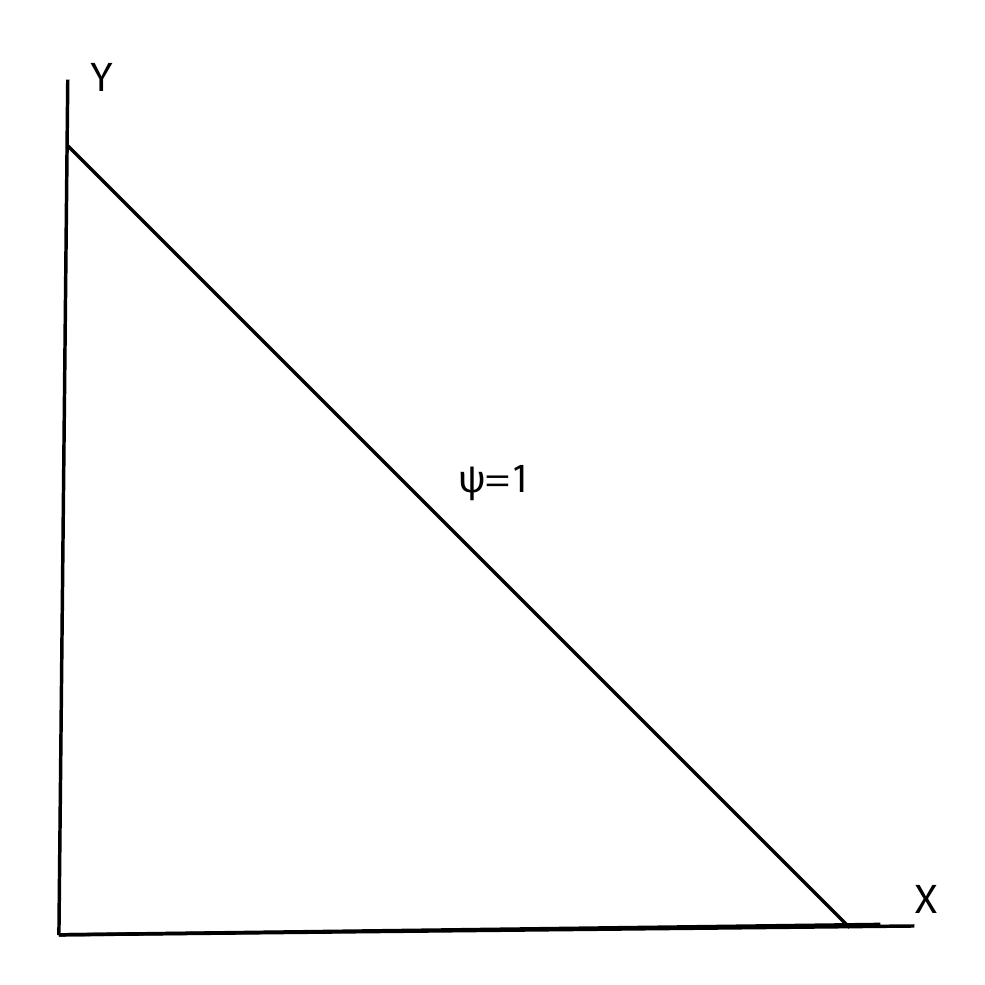}}
  \subfloat[$\rho \in (\tfrac{1}{2},1)$]{\label{figur:2}\includegraphics[width=70mm]{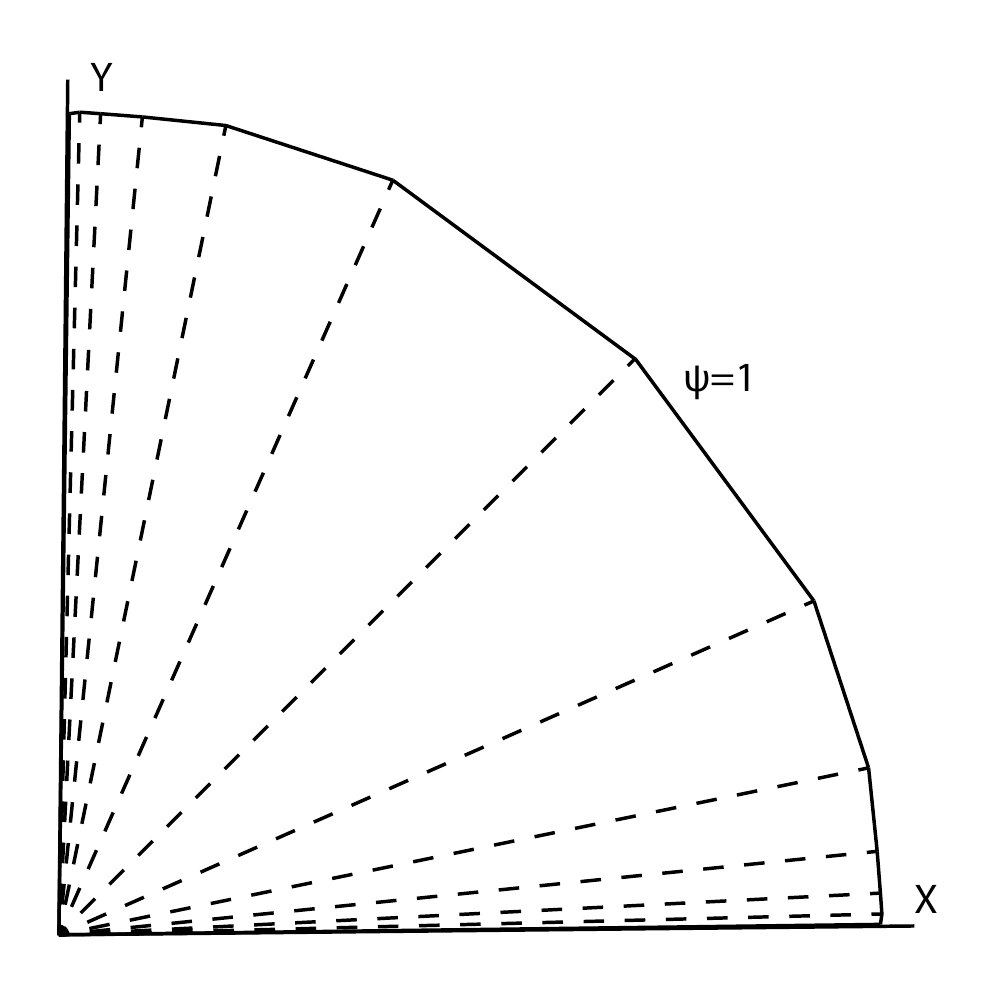}}
\end{figure}

\subsection{Properties of the limit metric}

It remains to study the structure of $\psi_{\rho}$. We show that it is piecewise affine outside of countably many lines of discontinuity. 

\begin{Lemma}\label{ndif}
The function $\psi_{\rho}$, restricted to points where $x > y > 0$, fails to be differentiable along the lines
\begin{equation*}
y = \frac{x}{k+1}, \; k \in \mathbb N,
\end{equation*}
and $y=x$, $y=0$. Furthermore, $\psi_{\rho}$ is piecewise affine.
\end{Lemma}

\begin{proof}
For each $(x,y)$ such that $x > y > 0$ there exists $k \in \mathbb N$ such that $1 \leq k \leq {x}/{y} < k + 1$, and therefore $k = \floor{{x}/{y}}, \text{ and } {x}/{(k+1)} \leq y < {x}/{k}$. Consequently, using \eqref{lengthhh}, the limit metric takes the form
\begin{align*}
\psi_{\rho}(x,y) &= y\ell_{\rho}(1) + \left(x - ky \right)\ell_{\rho}(k+1) + \left(y - x + ky \right)\ell_{\rho}(k)\\
& = \alpha(\rho,k) x + \beta(\rho,k) y,
\end{align*}
where we have set $\alpha(\rho,k) = \ell_{\rho}(k+1)- \ell_{\rho}(k)$ and $\beta(\rho,k) = \ell_{\rho}(1) + k \left( \ell_{\rho}(k) - \ell_{\rho}(k+1) \right) +  \ell_{\rho}(k)$. Clearly, on the set of points such that ${x}/({k+1}) < y < {x}/{k}$ it holds that $D\psi(x,y) = \left( \alpha(\rho,k) , \beta(\rho,k) \right) =: D\psi_k$. This demonstrates that outside of the lines $y = {x}/{(k+1)}, \; k \in \mathbb N$, $\psi_{\rho}$ is in fact affine. It therefore suffices to verify that the metric is not differentiable along these lines, that is, to show that for $k \in \mathbb N$ that $D\psi_k \neq D\psi_{k+1}$, for $k \in \mathbb N$. To this end 
\begin{align*}
\alpha(\rho,k+1) - \alpha(\rho,k) & = \ell_{\rho}(k+2)- \ell_{\rho}(k+1) - \left( \ell_{\rho}(k+1)- \ell_{\rho}(k) \right),\\
& = \int_{k+1}^{k+2} \frac{d \ell_{\rho}}{dx}(x) \d x -  \int_{k}^{k+1} \frac{d \ell_{\rho}}{dx}(x) \d x\\
& > \frac{d \ell_{\rho}}{dx}(k+1) - \frac{d \ell_{\rho}}{dx}(k+1) = 0,
\end{align*}
using the strict monotonicity of $d\ell_{\rho}/dx$ by \eqref{mono1}. The lines $y=x$ and $y=0$ follow with suitable modifications. 
\end{proof}

As a consequence of the piecewise affine structure, the following corollary also holds. 

\begin{corollary}
The level sets of $\psi_{\rho}$ are not strictly convex. 
\end{corollary}

\begin{remark}
The arguments of this paper can be easily adapted to the case where the region of higher length density is on rectangles rather than squares, provided the minimum side length is greater than $1/2$. A similar piecewise affine structure with infinitely many lines of discontinuity can be derived. The case when $\rho \leq \frac{1}{2}$ would need to be treated via different arguments, since the structure provided by Lemmas \ref{TLBR} and \ref{BRTL} no longer holds. Additionally, the case when $\beta \leq 2$ would require additional reasoning, an example such additional steps for the chessboard geometry can be found in \cite{amar09a}.
\end{remark}

\bibliographystyle{plain}
\bibliography{Refs}

\begin{thebibliography}{10}

\bibitem{acerbi84a}
E.~Acerbi and G.~Buttazzo.
\newblock On the limits of periodic {R}iemannian metrics.
\newblock {\em J. Analyse Math.}, 43:183--201, 1984.

\bibitem{amar09a}
M.~Amar, G.~Crasta, and A.~Malusa.
\newblock On the {F}insler metric obtained as limits of chessboard structures.
\newblock {\em Adv. Calc. Var.}, 2:321--360, 2009.

\bibitem{amar98a}
M.~Amar and E.~Vitali.
\newblock Homogenization of periodic {F}insler metrics.
\newblock {\em J. Convex Anal.}, 5(1):171--186, 1998.

\bibitem{braides98a}
A.~Braides and A.~Defranceschi.
\newblock {\em Homogenisation of multiple integrals}.
\newblock Oxford University Press, 1998.

\bibitem{buttazzo01a}
G.~Buttazzo, L.~{De Pascale}, and I.~Fragal\'a.
\newblock Topological equivalence of some variational problems involving
  distances.
\newblock {\em Discrete and Continuous Dynamical Systems}, 7(2):247--258, April
  2001.

\bibitem{lions09a}
P.~Cardaliaguet and P.-L. Lions.
\newblock A discussion about the homogenization of moving fronts.
\newblock {\em J. Math. Pure Appl.}, 91(4):339--363, 2009.

\bibitem{concordel97a}
M.~C. Concordel.
\newblock Periodic homogenisation of {H}amilton-{J}acobi equations. {II}.
  {E}ikonal equations.
\newblock {\em Proc. Roy. Soc. Edinburgh Sect. A}, 127:665--689, 1997.

\bibitem{c03}
B.~Craciun and K.~Bhattacharya.
\newblock {H}omogenisation of a {H}amilton-{J}acobi equation associated with
  the geometric motion of an interface.
\newblock {\em Proc. Roy. Soc. Edinburgh Sect. A}, 133A:773--805, 2003.

\bibitem{evans92a}
L.C. Evans.
\newblock Periodic homogenisation of certain fully nonlinear {PDE}.
\newblock {\em Proc. Roy. Soc. Edinburgh Sect. A}, 120:245--265, 1992.

\bibitem{gomes01a}
D.~Gomes and L.C. Evans.
\newblock Effective {H}amiltonians and averaging for {H}amiltonian dynamics
  {I}.
\newblock {\em Arch. Ration. Mech. Anal.}, (1), 2001.

\bibitem{gomes04a}
D.~Gomes and A.~Oberman.
\newblock Computing the effective {H}amiltonian: {A} varational approach to
  homogenization.
\newblock {\em SIAM J. Control Optim.}, 43(3):792--812, 2004.

\bibitem{jost05a}
J.~Jost.
\newblock {\em Riemannian Geometry and Geometric Analysis}.
\newblock Universitext. Springer-Verlag, fourth edition, 2005.

\bibitem{lions88a}
P.-L. Lions, G.~Papanicolaou, and S.R.S. Varadhan.
\newblock Homogenisation of {H}amilton-{J}acobi equations.
\newblock Preprint, 1988.

\bibitem{oberman09a}
A.~Oberman, R.~Takei, and A.~Vladimirsky.
\newblock Homogenisation of metric {H}amilton-{J}acobi equations.
\newblock {\em Multiscale Model. Simul.}, 8(2):269--295, 2009.

\bibitem{suttontese}
D.~C. Sutton.
\newblock {\em Macroscopic {H}amiltonian systems and their effective
  description}.
\newblock PhD thesis, University of {B}ath, 2013.

\end{thebibliography}

\end{document}